\documentclass{amsart}

\usepackage[all]{xypic}
\usepackage[centertags]{amsmath}
\usepackage{amsfonts}
\usepackage{amscd}
\usepackage{amssymb}
\usepackage{amsthm}
\usepackage{newlfont}
\usepackage{amsxtra}
 \newtheorem{thm}{Theorem}[section]
 \newtheorem{cor}[thm]{Corollary}
 
 \newtheorem{lem}[thm]{Lemma}
 \newtheorem{prop}[thm]{Proposition}
 \theoremstyle{definition}
 \newtheorem{defn}[thm]{Definition}
 \theoremstyle{remark}
 \newtheorem{rem}[thm]{Remark}
 \theoremstyle{remark}
 
 \theoremstyle{definition}
 
 \numberwithin{equation}{section}
 \newcommand{\Hom}{\mathrm{Hom}}
 
 \newcommand{\Spec}{\mathrm{Spec}}
 
 \newcommand{\Spf}{\mathrm{Spf}}

 \newcommand{\Aut}{\mathrm{Aut}}
 \newcommand{\End}{\mathrm{End}}

 \newcommand{\ord}{\mathrm{ord}}

 \newcommand{\Gal}{\mathrm{Gal}}

 \newcommand{\coker}{\mathrm{coker}}

 \newcommand{\Ram}{\mathrm{Ram}}

 \newcommand{\opp}{\mathrm{opp}}

 \newcommand{\Mass}{\mathrm{Mass}}
 \newcommand{\im}{\mathrm{Im}}

 \newcommand{\Id}{\mathrm{Id}}

 \newcommand{\alg}{\mathrm{alg}}

 \newcommand{\Fr}{\mathrm{Fr}}

 \newcommand{\inv}{\mathrm{inv}}
 \renewcommand{\mod}{\mathrm{mod}}
 \newcommand{\Nr}{\mathrm{Nr}}
 \newcommand{\fp}{\mathfrak p}
 
 \newcommand{\fn}{\mathfrak n}

 \newcommand{\fX}{\mathfrak X}

 \newcommand{\cO}{\mathcal{O}}

 \newcommand{\cK}{\mathcal{K}}

 \newcommand{\cA}{\mathcal{A}}
 \newcommand{\cB}{\mathcal{B}}

 \renewcommand{\cR}{\mathcal{R}}
 \renewcommand{\cD}{\mathcal{D}}
 \newcommand{\cE}{\mathcal{E}}

 \newcommand{\cI}{\mathcal{I}}
 
 \renewcommand{\cH}{\mathcal{H}}
 \newcommand{\cT}{\mathcal{T}}
 \newcommand{\C}{\mathbb{C}}
 \newcommand{\E}{\mathbb{E}}
 \newcommand{\F}{\mathbb{F}}
 \newcommand{\M}{\mathbb{M}}
 \newcommand{\Q}{\mathbb{Q}}
 \newcommand{\Z}{\mathbb{Z}}
 \newcommand{\A}{\mathbb{A}}

 \newcommand{\bF}{\mathbf{f}}

 \newcommand{\Ell}{\mathcal{E}\ell\ell}
 \newcommand{\ctimes}{\widehat{\otimes}}

 \newcommand{\bs}{\setminus}

 \newcommand{\G}{\Gamma}
 
 \newcommand{\La}{\Lambda}
 \newcommand{\la}{\lambda}

 \newcommand{\comr}[1]{[\![#1]\!]}
 \newcommand{\comf}[1]{(\!(#1)\!)}
 \newcommand{\twist}[1]{{^\tau}\!#1}

\begin{document}

\title{Endomorphisms of exceptional $\cD$-elliptic sheaves}

\author{Mihran Papikian}

\address{Department of Mathematics, Pennsylvania State University, University Park, PA 16802}

\email{papikian@math.psu.edu}


\thanks{The author was supported in part by NSF grant DMS-0801208 and Humboldt Research Fellowship.}

\date{}
% ----------------------------------------------------------------------

\begin{abstract}
We relate the endomorphism rings of certain $\cD$-elliptic sheaves
of finite characteristic to hereditary orders in central division
algebras over function fields.
\end{abstract}

% ----------------------------------------------------------------------
\maketitle
% ----------------------------------------------------------------------

%\section*{}

% ----------------------------------------------------------------------

\section{Introduction} The endomorphism rings of abelian varieties have long
been a subject of intensive investigation in number theory. One of
the earliest results in this area was the determination by Deuring
of the endomorphism rings of elliptic curves over finite fields. His
results were later generalized to higher dimensional abelian
varieties by Honda, Tate and Waterhouse \cite{Waterhouse}. These
results have important applications, e.g., they play a key role in
calculations of local zeta functions of Shimura varieties.

In \cite{Drinfeld}, Drinfeld introduced a certain function field
analogue of abelian varieties; these objects are now called Drinfeld
modules. Denote by $\F_q$ the finite field with $q$ elements. Let
$X$ be a smooth, projective, geometrically connected curve defined
over $\F_q$. Let $F=\F_q(X)$ be the function field of $X$. Fix a
place $\infty$ of $F$ (in Drinfeld's theory this plays the role of
an archimedean place). Let $o\neq \infty$ be another place of $F$.
Denote by $\F_o$ the residue field at $o$. In \cite{Drinfeld2},
Drinfeld proved the analogue of Honda-Tate theorem for Drinfeld
modules defined over extensions of $\F_o$. In \cite{GekelerJA},
Gekeler extended Drinfeld's results, in particular, he proved that
for a rank-$d$ supersingular Drinfeld module $\phi$ over
$\overline{\F}_o$ the endomorphism ring $\End(\phi)$ is a maximal
order in the central division algebra over $F$ of dimension $d^2$,
which is ramified exactly at $o$ and $\infty$ with invariants $-1/d$
and $1/d$, respectively. Moreover, there is a bijection between the
isomorphism classes of rank-$d$ supersingular Drinfeld modules over
$\overline{\F}_o$ and the left ideal classes of $\End(\phi)$ (see
\cite[Thm. 4.3]{GekelerJA}).

In \cite{LRS}, Laumon, Rapoport and Stuhler introduced the notion of
$\cD$-elliptic sheaves, which is a generalization of the notion of
Drinfeld modules. (One can think of these objects as function field
analogues of abelian varieties equipped with an action of a maximal
order in a simple algebra over $\Q$.) In Laumon-Rapoport-Stuhler
theory one needs to fix a central simple algebra $D$ over $F$ of
dimension $d^2$ which is split at $\infty$, and a maximal
$\cO_X$-order $\cD$ in $D$ (see $\S$\ref{sOrders} for definitions).
In \cite[Ch. 9]{LRS}, the authors develop the analogue of Honda-Tate
theory for $\cD$-elliptic sheaves over $\overline{\F}_o$ with zero
$o$ and pole $\infty$, assuming $D$ is split at $o$. The assumption
that $D$ is split at $o$ is not superficial. When $D$ is ramified at
$o$, to obtain a reasonable theory of $\cD$-elliptic sheaves with
zero $o$ and pole $\infty$, one has to assume at least that
$D\otimes_F F_o$ is the division algebra with invariant $1/d$ over
$F_o$, where $F_o$ is the completion of $F$ at $o$. Such
$\cD$-elliptic sheaves play a crucial role in the function field
analogue of \u{C}erednik-Drinfeld uniformization theory developed by
Hausberger \cite{Hausberger}.

Assume $D_o:=D\otimes_F F_o$ is the $d^2$-dimensional central
division algebra with invariant $1/d$ over $F_o$. In this paper we
define a subclass of $\cD$-elliptic sheaves over $\overline{\F}_o$,
which we call \textit{exceptional}, and which are distinguished by a
particularly simple relationship between the actions of $D_o$ and
the Frobenius at $o$; see Definition \ref{def3.1}. In general,
exceptional $\cD$-elliptic sheaves do not correspond to points on
the moduli schemes constructed in \cite{LRS} or \cite{Hausberger},
so they are not very natural from moduli-theoretic point of view.
Nevertheless, we show that the theory of endomorphism rings of these
objects is similar to the theory of endomorphism rings of
supersingular Drinfeld modules. The main result is the following
(see Theorems \ref{thm2.2} and \ref{thm2.3}):

\begin{thm}\label{thm-main}
Let $\E$ be an exceptional $\cD$-elliptic sheaf over
$\overline{\F}_o$ of type $\bF$. Then $\End(\E)$ is a hereditary
$\cO_X$-order in the central division algebra $\bar{D}$ over $F$
with invariants
$$
\inv_x(\bar{D})=\left\{
  \begin{array}{ll}
    1/d, & x=\infty; \\
    0, & x=o; \\
    \inv_x(D), & x\neq o,\infty.
  \end{array}
\right.
$$
This order is maximal at every place $x\neq o$, and at $o$ it is
isomorphic to a hereditary order of type $\bF$. There is a bijection
between the set of isomorphism classes of exceptional $\cD$-elliptic
sheaves over $\overline{\F}_o$ of type $\bF$ and the isomorphism
classes of locally free rank-$1$ right $\End(\E)$-modules.
\end{thm}

The type of an exceptional $\cD$-elliptic sheaf is determined by the
action of the Frobenius at $o$, and the type of a hereditary order
determines the order up to an isomorphism (see $\S$\ref{sOrders}).
In $\S$\ref{sMass}, we use Theorem \ref{thm-main} to prove a
mass-formula for exceptional $\cD$-elliptic sheaves, and discuss a
geometric application of this formula. In $\S$\ref{sLast}, we
explain how the argument in the proof of Theorem \ref{thm-main} can
be used to prove a theorem about endomorphism rings of supersingular
$\cD$-elliptic sheaves over $\overline{\F}_o$, which implies
Gekeler's result mentioned earlier as a special case (in
$\S$\ref{sLast} we assume that $D$ is split at $o$).

\vspace{0.1in}

\noindent\textbf{Notation.} Unless specified otherwise, the
following notation is fixed throughout the article.

$\bullet$ $k$ is a fixed algebraic closure of $\F_q$ and $\Fr_q:
k\to k$ is the automorphism $x\mapsto x^q$.

$\bullet$ $|X|$ denotes the set of closed points on $X$ (equiv. the
set of places of $F$).

$\bullet$ For $x\in |X|$, $\cO_x$ is the completion of $\cO_{X,x}$,
and $F_x$ (resp. $\F_x$) is the fraction field (resp. the residue
field) of $\cO_x$. The \textit{degree} of $x$ is
$\deg(x):=[\F_x:\F_q]$, and $q_x:=q^{\deg(x)}=\#\F_x$. We fix a
uniformizer $\pi_x$ of $\cO_x$.

$\bullet$ $\A_F:=\prod_{x\in |X|}'F_x$ denotes the adele ring of
$F$, and for a set of places $S\subset |X|$, $\A_F^S:=\prod_{x\in
|X|-S}'F_x$ denotes the adele ring outside of $S$.

$\bullet$ The \textit{zeta-function of $X$} is
$$
\zeta_X(s)=\prod_{x\in |X|}(1-q_x^{-s})^{-1}, \quad s\in \C.
$$

$\bullet$ For any ring $R$ we denote by $R^\times$ its subgroup of
units.

$\bullet$ $\M_d$ denotes the ring of $d\times d$ matrices.

\section{Orders}\label{sOrders} For the convenience of the reader, we recall some
basic definitions and facts concerning orders over Dedekind domains.
A standard reference for these topics is \cite{Reiner}.

Let $R$ be a Dedekind domain with quotient field $K$ and let $A$ be
a central simple $K$-algebra. For any finite dimensional $K$-vector
space $V$, a \textit{full $R$-lattice} in $V$ is a finitely
generated $R$-submodule $M$ in $V$ such that $K\otimes_R M\cong V$.
An \textit{$R$-order} in the $K$-algebra $A$ is a subring $\La$ of
$A$, having the same unity element as $A$, and such that $\La$ is a
full $R$-lattice in $A$. A \textit{maximal $R$-order} in $A$ is an
$R$-order which is not contained in any other $R$-order in $A$. A
\textit{hereditary $R$-order} in $A$ is an $R$-order $\La$ which is
a hereditary ring, i.e., every left (equiv. right) ideal of $\La$ is
a projective $\La$-module. Maximal orders are hereditary. Being
maximal or hereditary are local properties for orders: an $R$-order
$\La$ in $A$ is maximal (resp. hereditary) if and only if
$\La_\fp:=\La\otimes_R R_\fp$ is a maximal (resp. hereditary)
$R_\fp$-order in $A_\fp:=A\otimes_K K_\fp$ for all prime ideals
$\fp\lhd R$, where $R_\fp$ and $K_\fp$ are the $\fp$-adic
completions of $R$ and $K$.

Let $I$ be a full $R$-lattice in $A$. Define the \textit{left order
of $I$}
$$
O_\ell(I)=\{a\in A\ |\ aI\subseteq I\}.
$$
It is easy to see that $O_\ell(I)$ is an $R$-order in $A$. One
similarly defines the right order $O_r(I)$ of $I$.

\vspace{0.1in}

Assume $R$ is a complete discrete valuation ring with a uniformizer
$\pi$ and fraction field $K$. Let $\bF=(f_0,\dots, f_{d-1})$ be a
$d$-tuple of non-negative integers such that $\sum_{i=0}^{d-1}
f_i=d$. Denote by $\M_d(\bF,R)$ the subgroup of $\M_d(R)$ consisting
of matrices of the form $(m_{ij})$, where $m_{ij}$ ranges over all
$f_i\times f_j$ matrices with entries in $R$ if $i\geq j$, and over
all $f_i\times f_j$ matrices with entries in $\pi R$ if $i< j$ (a
block of size $0$ is assumed to be empty), e.g., if
$\bF=(d,0,\dots,0)$ then $\M_d(\bF,R)=\M_d(R)$.

\begin{thm}\label{thm-her} Let $\La$ be an
$R$-order in $\M_d(K)$. $\La$ is maximal if and only if there is an
invertible element $u\in \M_d(K)$ such that $u\La u^{-1}=\M_d(R)$;
$\La$ is hereditary if and only if $u\La u^{-1}=\M_d(\bF,R)$ for
some $\bF$ (uniquely determined up to permutation of its entries).
\end{thm}
\begin{proof}
See Theorems 17.3 and 39.14 in \cite{Reiner}.
\end{proof}

When $\La$ is a hereditary order as in Theorem \ref{thm-her}, we
shall call $\bF$ the \textit{type} of $\La$. (This is slightly
different from the terminology used in \cite[p. 360]{Reiner}.)

\vspace{0.1in}

Let $A$ be a central simple algebra over $F$. An
\textit{$\cO_X$-order in $A$} is a coherent locally  free sheaf
$\cA$ of $\cO_X$-algebras with generic fibre $A$. The $\cO_X$-order
$\cA$ is \textit{maximal} (resp. \textit{hereditary}) if for every
open affine $U=\Spec(R)\subset X$ the set of sections $\cA(U):=\G(U,
\cA)$ is a maximal (resp. hereditary) $R$-order in $A$. For $x\in
|X|$ we denote $A_x:=A\otimes_F F_x$ and
$\cA_x:=\cA\otimes_{\cO_{X,x}}\cO_x$, so $\cA_x$ is isomorphic to a
subring of $A_x$. $\cA$ is maximal (resp. hereditary) if and only if
$\cA_x$ is a maximal (resp. hereditary) $\cO_x$-order in $A_x$.

Let $\cA$ be a hereditary $\cO_X$-order. Let $\cI$ be a coherent
sheaf on $X$ which is a locally free rank-$1$ right $\cA$-module.
(The action of $\cA$ on $\cI$ extends the action of $\cO_X$.) The
generic fibre $\cI\otimes_{\cO_X}F$ is isomorphic to $A$ as an
$F$-vector space. Define a sheaf $O_\ell(\cI)$ on $X$ as follows.
For an open affine $U\subset X$ let
$$
O_\ell(\cI)(U)=O_\ell(\cI(U)).
$$
It is easy to see that $O_\ell(\cI)$ is an $\cO_X$-order in $A$,
which is locally isomorphic to $\cA$.

\section{Dieudonn\'e modules}\label{sec1}

Let $R$ be a complete discrete valuation ring of positive
characteristic $p>0$ and residue field $\F_q$. Fix a uniformizer
$\pi$ of $R$ and identify $R=\F_q\comr{\pi}$. Let $K$ be the
fraction field of $R$. Let $\cR=R\ctimes_{\F_q}k\cong k\comr{\pi}$
be the completion of the maximal unramified extension of $R$, and
$\cK=K\ctimes_{\F_q}k\cong k\comf{\pi}$ be the field of fractions of
$\cR$. We will denote the canonical lifting of $\Fr_q\in
\Gal(k/\F_q)$ to $\Aut(\cK)$ by the same symbol, so
$$
\Fr_q\left(\sum_{i=n}^\infty a_i \pi^i\right)=\sum_{i=n}^\infty
a_i^q \pi^i,\quad n\in \Z.
$$

The following definition and Theorem \ref{simpleN} below are due to
Drinfeld \cite{DrinfeldPC}; see also \cite[$\S$2.4]{Laumon}.

\begin{defn}
A \textit{Dieudonn\'e $R$-module over $k$} is a free $\cR$-module of
finite rank $M$ endowed with an injective $\Fr_q$-linear map
$\varphi: M\to M$ such that the cokernel of $\varphi$ is finite
dimensional as a $k$-vector space. The \textit{rank} of
$(M,\varphi)$ is the rank of $M$ as a $\cR$-module. A
\textit{Dieudonn\'e $K$-module over $k$} is a finite dimensional
$\cK$-vector space $N$ endowed with a bijective $\Fr_q$-linear map
$\varphi: N\to N$. The \textit{rank} of $(N,\varphi)$ is the
dimension of $N$ as a $\cK$-vector space. A \textit{morphism} of
Dieudonn\'e $R$-modules (resp. $K$-modules) over $k$ is a linear map
between the underlying $\cR$-modules (resp. $\cK$-vector spaces)
which commutes with the $\Fr_q$-linear maps $\varphi$. If
$(M,\varphi)$ is a Dieudonn\'e $R$-module over $k$, then
$(K\otimes_R M, K\otimes_R \varphi)$ is a Dieudonn\'e $K$-module
over $k$.
\end{defn}

Let $\cK\{\tau\}$ be the non-commutative polynomial ring with
commutation rule $\tau\cdot a=\Fr_q(a)\tau$, $a\in \cK$. For each
pair of integers $(r,s)$ with $r\geq 1$ and $(r,s)=1$, let
$$
N_{r,s}=\cK\{\tau\}/\cK\{\tau\}(\tau^r-\pi^s).
$$
Then $(N_{r,s}, \varphi_{r,s})$, where $\varphi_{r,s}$ is the left
multiplication by $\tau$, is a Dieudonn\'e $K$-module over $k$ of
rank $r$.

\begin{thm}\label{simpleN}
The category of Dieudonn\'e $K$-modules over $k$ is $K$-linear and
semi-simple. Its simple objects are $(N_{r,s},\varphi_{r,s})$,
$r,s\in \Z$, $r\geq 1$, $(r,s)=1$. The $K$-algebra of endomorphisms
$D_{r,s}=\End(N_{r,s},\varphi_{r,s})$ of such an object is the
central division algebra over $K$ with invariant $-s/r$.
\end{thm}

This theorem implies that given a Dieudonn\'e $K$-module
$(N,\varphi)$, its endomorphism algebra $\End(N,\varphi)$ is a
finite dimensional semi-simple $K$-algebra such that the center of
each simple component is $K$. It is clear that for a Dieudonn\'e
$R$-module $(M,\varphi)$, the endomorphism ring $\End(M,\varphi)$ is
an $R$-order in $\End(N,\varphi)$, where $(N,\varphi)=K\otimes
(M,\varphi)$.

\begin{prop}\label{lemEtale} Let $(M,\varphi)$ be a Dieudonn\'e
$R$-module over $k$ of rank $n$. Suppose $\varphi(M)=M$. Then
$$
(M,\varphi)\cong (R^n\ctimes_{\F_q}k, \Id\ctimes_{\F_q}\Fr_q).
$$
\end{prop}
\begin{proof} This is proven in \cite[Prop. 2.5]{DrinfeldPC}
using $\pi$-divisible groups. An alternative argument is as follows.
Let $(N,\varphi)$ be the associated Dieudonn\'e $K$-module. By
\cite[Prop. 2.4.6]{Laumon}, the assumption of the proposition is
equivalent to $$(N,\varphi)\cong (N_{1,0},\varphi_{1,0})^n.
$$
Hence $$(N,\varphi)\cong(N^\varphi\ctimes_{\F_q}k,
\Id\ctimes_{\F_q}\Fr_q),$$ where $N^\varphi=\{a\in N\ |\
\varphi(a)=a\}$ is an $n$-dimensional $K$-vector space. Since
$N=M\otimes K$, $M^\varphi:=M\cap N^\varphi$ is a full $R$-lattice
in $N^\varphi$ and $M=M^\varphi\ctimes_{\F_q} k$.
\end{proof}

Let $D$ be the $d^2$-dimensional central division algebra over $K$
with invariant $1/d$. Let $\cD$ be the maximal $R$-order in $D$.
Denote by $R_d=\F_{q^d}\comr{\pi}$ the ring of integers of the
degree $d$ unramified extension of $K$. We can identify $\cD$ with
the $R$-algebra $R_d\comr{\Pi}$ of non-commutative formal power
series in the indeterminate $\Pi$ satisfying the relations
\begin{align*}
\Pi a &=\Fr_q(a) \Pi\quad \text{for any } a\in R_d, \\
\Pi^d &=\pi.
\end{align*}

\begin{defn} A Dieudonn\'e $R$-module $(M,\varphi)$ over $k$ is
\textit{connected} if for all large enough positive integers $m$,
$\varphi^m(M)\subset \pi M$. This is equivalent to saying that
$(N_{1,0}, \varphi_{1,0})$ does not appear in the decomposition of
the associated Dieudonn\'e $K$-module into simple factors. A
\textit{Dieudonn\'e $\cD$-module over $k$} is a rank-$d^2$ connected
Dieudonn\'e $R$-module $(M,\varphi)$ over $k$ equipped with a right
$\cD$-action which commutes with $\varphi$ and extends the natural
action of $R$. A \textit{morphism} of Dieudonn\'e $\cD$-modules is a
morphism of the underlying Dieudonn\'e $R$-modules which commutes
with the action of $\cD$. For each Dieudonn\'e $\cD$-module
$(M,\varphi)$ over $k$ there is an associated Dieudonn\'e $D$-module
$(N,\varphi)=K\otimes(M,\varphi)$.
\end{defn}

A Dieudonn\'e $\cD$-module $(M,\varphi)$ over $k$ is naturally a
right $\cD\ctimes_{\F_q} k$-module. Fix an embedding
$\F_{q^d}\hookrightarrow k$. Since $\F_{q^d}$ is also embedded in
$\cD$, we obtain a grading
$$
M=\bigoplus_{i\in \Z/d\Z} M_i,
$$
where $M_i=\{m\in M\ |\ m(\la\ctimes 1)=m(1\ctimes \la^{q^i}),
\la\in \F_{q^d} \}$. Each $M_i$ is a free finite rank $\cR$-module.
The action of $\Pi\ctimes 1$ on $M$ induces injective linear maps
$\Pi_i:M_i\to M_{i+1}$, $i\in \Z/d\Z$. The composition
$$
M_i\xrightarrow{\Pi_i}M_{i+1}\xrightarrow{\Pi_{i+1}}\cdots
\xrightarrow{\Pi_{i+d-1}}M_{i+d}=M_i
$$
is $M_i(\pi\ctimes 1)=\pi M_i$. In particular, all $M_i$ have the
same rank over $\cR$, which must be $d$, since the rank of $M$ is
$d^2$. Since $\dim_k(\coker(\pi))=d^2$, we have
$\dim_k(\coker(\Pi))=d$.

Similarly, $\varphi$ induces injective $\Fr_q$-linear maps
$$
M_i\xrightarrow{\varphi_i}M_{i+1}\xrightarrow{\varphi_{i+1}}\cdots
\xrightarrow{\varphi_{i+d-1}}M_{i+d}=M_i.
$$
Let $f_i:=\dim_k(\coker(\varphi_i))$, so $\sum_{i=0}^{d-1}
f_i=\dim_{k}(\coker(\varphi))$. We call the ordered $d$-tuple
$\bF=(f_0,\dots, f_{d-1})$ the \textit{type} of $M$.

\begin{defn}\label{defnExp} (cf. \cite{Ribet}, \cite{Genestier})
A Dieudonn\'e $\cD$-module $(M,\varphi)$ over $k$ is
\textit{exceptional} if $\im(\varphi)=\im(\Pi)$. $(M,\varphi)$ is
\textit{special} if $f_i=1$ for all $i$. $(M,\varphi)$ is
\textit{superspecial} if it is special and exceptional.
\end{defn}

Note that $(M,\varphi)$ being exceptional is equivalent to
$\im(\varphi_i)=\im(\Pi_i)$ for all $i\in \Z/d\Z$, i.e., every index
of $M$ is \textit{critical} in the terminology of \cite[Def.
II.1.3]{Genestier}. In particular, if $(M,\varphi)$ is exceptional
of type $\bF$ then $\sum_{i=0}^{d-1} f_i=d$.

\begin{prop}\label{propEndM}
Let $(M,\varphi)$ be an exceptional Dieudonn\'e $\cD$-module over
$k$ of type $\bF$. Then $\End_\cD(M,\varphi)\cong \M_d(\bF,R)$. In
particular, $\End_\cD(M,\varphi)$ is a hereditary $R$-order in
$\End_D(N,\varphi)\cong \M_d(K)$.
\end{prop}
\begin{proof}
Using the injections $\Pi_i$, we can identify all $M_i\otimes K$
with the same $d$-dimensional $\cK$-vector space $V$. Then $\Pi_i$
induces a bijective linear map $V\to V$, and $\varphi_i$ induces a
bijective $\Fr_q$-linear map. Consider $\Pi_i^{-1}\circ
\varphi_i:V\to V$ as a bijective $\Fr_q$-linear map. Since
$(M,\varphi)$ is exceptional, $\im(\Pi_i)=\im(\varphi_i)$ for all
$i\in \Z/d\Z$. Hence $\Pi_i^{-1}\circ \varphi_i$ is bijective on
$M_i$. By Proposition \ref{lemEtale}, there are $d$ full
$R$-lattices $\La_i$ in $K^d$, $i\in \Z/d\Z$, such that
$M_i=\La_i\ctimes k$. Since the action of $\cD$ commutes with the
action of $\varphi$, we have $\varphi_{i+1}\circ
\Pi_i=\Pi_{i+1}\circ \varphi_i$. This implies that the
identifications $M_i=\La_i\ctimes k$ can be made compatibly so that
\begin{equation}\label{eq-flag}
\La_0 \xrightarrow{\pi_0} \La_1 \xrightarrow{\pi_1} \cdots
\xrightarrow{\pi_{d-2}} \La_{d-1}\xrightarrow{\pi_{d-1}}  \La_0,
\end{equation}
where $\pi_i$'s are injections, $\Pi_i=\pi_i\ctimes k$,
$\varphi_i=\pi_i\ctimes \Fr_q$. The inclusions $\pi_i$ satisfy
$$
\pi_i\circ\pi_{i+1}\circ\cdots\circ \pi_{i+d-1}=\pi,
$$
so each
$\coker(\pi_i)$ has no nilpotents and
$\dim_{\F_q}(\coker(\pi_i))=f_i$.

Now it is easy to see that giving an endomorphism of $(M,\varphi)$
commuting with the action of $\cD$ is equivalent to giving an
endomorphism $g$ of $K^d$ which preserves the flag of lattices
(\ref{eq-flag}). Such endomorphisms form an $R$-algebra isomorphic
to $\M_d(\bF,R)$. That this is a hereditary order in $\M_d(K)$
follows from Theorem \ref{thm-her}.
\end{proof}

Every Dieudonn\'e $\cD$-module over $k$ satisfies the properties in
\cite[p. 20]{Genestier}, hence corresponds to a formal $\cD$-module
of height $d^2$. It is instructive to give explicit examples of such
formal modules. What follows below is motivated by
\cite[I.4.2]{Genestier}.

The underlying formal group is isomorphic to
$\widehat{\mathbb{G}}_{a,k}^d$, where
$\widehat{\mathbb{G}}_{a,k}=\Spf(k\comr{t})$ is the formal additive
group. Denote by $\tau$ the Frobenius isogeny of
$\widehat{\mathbb{G}}_{a,k}$ corresponding to $t\mapsto t^q$. To
give a formal $\cD$-module essentially amounts to giving an
embedding
$$\Phi: \F_{q^d}\comr{\Pi}=\cD \hookrightarrow \End(\mathbb{G}_{a,k}^d)\cong \M_d(k\{\!\{\tau\}\!\}),$$
where $k\{\!\{\tau\}\!\}$ is the non-commutative ring of formal
power series in $\tau$ satisfying $\tau a=a^q\tau$ for $a\in k$.

Now let $(M,\varphi)$ be an exceptional Dieudonn\'e $\cD$-module of
type $\bF$. Being exceptional, i.e., $\im(\Pi)=\im(\varphi)$,
translates into $$\Phi(\Pi)=\tau\cdot \Id.$$ The type translates
into
$$\Phi(\la)=\mathrm{diag}(\chi_{ij}(\la))_{0\leq j\leq d-1, 1\leq i\leq
f_j}, \ \la\in \F_{q^d},$$ where $\chi_{ij}(\la)=\la^{q^j}$ if
$f_j\neq 0$ and is omitted from $\mathrm{diag}(\cdot)$ otherwise.
For example, if $d=3$ and $\bF=(2,0,1)$ then
$$
\Phi(\Pi)=\begin{pmatrix} \tau & 0 & 0 \\ 0 &\tau & 0\\ 0 & 0 & \tau
\end{pmatrix}\quad \text{and}\quad \Phi(\la)=\begin{pmatrix} \la & 0 & 0 \\ 0 &\la & 0\\ 0 & 0 &
\la^{q^2}
\end{pmatrix}.
$$
The endomorphism ring $\End_\cD(M,\varphi)$ is isomorphic to the
opposite algebra of the centralizer of $\Phi(\cD)$ in
$\M_d(k\{\!\{\tau\}\!\})$. One can check as in
\cite[I.4.2]{Genestier} that this centralizer is isomorphic to
$\M_d(\bF,R)^\opp$.

\section{$\cD$-elliptic sheaves}\label{Sec2} In this section we
recall the definition of $\cD$-elliptic sheaves of finite
characteristic and their basic properties as given in \cite[Ch.
9]{LRS}.

Fix a closed point $\infty\in |X|$. Let $D$ be a central simple
algebra over $F$ of dimension $d^2$. Assume $D$ is split at
$\infty$, i.e., $D\otimes_F F_\infty\cong \M_d(F_\infty)$. Fix a
maximal $\cO_X$-order $\cD$ in $D$. Denote by $\Ram\subset |X|$ the
set of places where $D$ is ramified; hence for all $x\not\in \Ram$
the couple $(D_x,\cD_x)$ is isomorphic to $(\M_d(F_x),
\M_d(\cO_x))$. Fix another closed point $o \in |X|-\infty$, and an
embedding $\F_o\hookrightarrow k$. Let $z$ be the morphism
determined by these choices
$$
z:\Spec(k)\to \Spec(\F_o)\hookrightarrow X.
$$

\begin{defn} A \textit{$\cD$-elliptic sheaf of characteristic $o$ over $k$}
is a sequence $\E=(\cE_i,j_i,t_i)_{i\in \Z}$, where $\cE_i$ is a
locally-free $\cO_{X\otimes_{\F_q}k}$-module of rank $d^2$, equipped
with a right action of $\cD$ which extends the $\cO_X$-action, and
\begin{align*}
j_i &:\cE_i\hookrightarrow \cE_{i+1}\\
t_i &:\twist{\cE_i}:=(\Id_X\otimes \Fr_q)^\ast \cE_i\hookrightarrow
\cE_{i+1}
\end{align*}
are injective $\cD$-linear homomorphisms. Moreover, for each $i\in
\Z$ the following conditions hold:
\begin{enumerate}
\item The diagram
$$
\xymatrix{\cE_i \ar[r]^{j_i} & \cE_{i+1}\\ \twist{\cE_{i-1}}
\ar[r]^{\twist{j_{i-1}}}\ar[u]^-{t_{i-1}} &
\twist{\cE_i}\ar[u]_-{t_i}}
$$
commutes;
\item $\cE_{i+d\cdot\deg(\infty)}=\cE_i\otimes_{\cO_X}\cO_X(\infty)$, and the inclusion
$$
\cE_i\xrightarrow{j_i}\cE_{i+1}\xrightarrow{j_{i+1}}\cdots \to
\cE_{i+d\cdot\deg(\infty)}=\cE_i\otimes_{\cO_X}\cO_X(\infty)
$$
is induced by $\cO_X\hookrightarrow \cO_X(\infty)$;
\item $\dim_kH^0(X\otimes k,\coker j_i)=d$;
\item $\cE_i/t_{i-1}(\twist{\cE_{i-1}})=z_\ast\cH_i$, where $\cH_i$
is a $d$-dimensional $k$-vector space.
\end{enumerate}
\end{defn}

When $o\not\in \Ram$, the previous definition is exactly the one
found in \cite[p. 260]{LRS}. When $o\in \Ram$ this definition is not
restrictive enough, but for our purposes it is adequate to take it
as a starting point.

\begin{defn} Let $\mathbf{DES}$ be the category whose objects are
the $\cD$-elliptic sheaves of characteristic $o$ over $k$, and a
morphism between two objects in this category
$$
\psi=(\psi_i)_{i\in \Z}:\E'=(\cE_i', j_i', t_i')_{i\in \Z}\to
\E''=(\cE_i'', j_i'', t_i'')_{i\in \Z}
$$
is a sequence of sheaf morphisms $\psi_i:\cE_i'\to \cE''_i$ which
are compatible with the action of $\cD$ and commute with the
morphisms $j_i$ and $t_i$:
$$
\psi_{i+1}\circ j_i'=j_i''\circ \psi_i \quad \text{and}\quad
\psi_i\circ t_{i-1}'=t_i''\circ \twist{\psi_{i-1}}.
$$
Denote by $\Hom(\E',\E'')$ the set of all morphisms $\E'\to \E''$,
and let $\End(\E)=\Hom(\E,\E)$.
\end{defn}

\begin{defn}[\cite{DrinfeldPC}]
A \textit{$\varphi$-space} over $k$ is a finite dimensional
$F\otimes_{\F_q} k$-vector space $N$ equipped with a bijective
$F\otimes_{\F_q}\Fr_q$-linear map $\varphi:N\to N$. A
\textit{morphism} $\alpha$ between two $\varphi$-spaces
$(N',\varphi')$ and $(N'',\varphi'')$ is a $F\otimes_{\F_q}k$-linear
map $N'\xrightarrow{\alpha}N''$ such that $\varphi''\circ
\alpha=\alpha\circ \varphi'$.
\end{defn}

Let $\E\in \mathbf{DES}$. Denote $N=H^0(\Spec(F\otimes_{\F_q} k),
\cE_0)$. This is a free $D\otimes_{\F_q}k$-module of rank $1$. The
$t_i$'s induce a bijective $F\otimes_{\F_q} \Fr_q$-linear map
$\varphi:N\to N$, compatible with the action of $D$ on the right, so
to $\E$ one can attach a $\varphi$-space over $k$ equipped with an
action of $D$, which commutes with $\varphi$. This action induces an
$F$-algebra homomorphism
$$
\iota: D^\opp\to \End(N,\varphi).
$$
We denote by $\End_D(N,\varphi)$ the $F$-algebra of endomorphisms of
$(N,\varphi)$ which commute with the action of $D$. The triple
$(N,\varphi,\iota)$ is called the \textit{generic fibre of $\E$}
(\cite[Def. 9.2]{LRS}). It is independent of the choice of $\cE_0$
since the sheaves $\cE_i$ are isomorphic over $(X-\infty)\otimes k$
via $j$'s.

For $x\in |X|$, denote $M_x:=H^0(\Spec(\cO_x\ctimes k), \cE_0)$.
This is a free $\cO_x\ctimes_{\F_q} k$-module of rank $d^2$ with a
right action of $\cD_x$. Let $N_x=F_x\otimes_{\cO_x} M_x$. The
$t_i$'s induce a bijective $F_x\ctimes_{\F_q} \Fr_q$-linear map
$\varphi_x: N_x\to N_x$, compatible with the action of $D_x$. The
pair $(N_x,\varphi_x)$ is the \textit{Dieudonn\'e module} of $\E$ at
$x$. The $F_x$-algebra of endomorphisms of $(N_x,\varphi_x)$ which
commute with the action of $D_x$ will be denoted by $\End_{D_x}(N_x,
\varphi_x)$. Note that $(N_x,\varphi_x)=(F_x\ctimes_F N,
F_x\ctimes_F \varphi)$.

As easily follows from definitions, the lattices $M_x$ have the
following properties (see \cite[Lem. 9.3]{LRS}):
\begin{enumerate}
\item[(\textbf{M1})] If $x=\infty$, then
\begin{align*}
M_\infty\subset \varphi_\infty(M_\infty)\\
\dim_k(\varphi_\infty(M_\infty)/M_\infty)=d\\
\varphi_\infty^{d\cdot\deg(\infty)}(M_\infty)=
\pi_\infty^{-1}M_\infty.
\end{align*}
\item[(\textbf{M2})] If $x=o$, then
$$
\pi_oM_o\subset \varphi_o(M_o)\subset M_o
$$
and the $\F_o\otimes_{\F_q} k$-module $M_o/\varphi_o(M_o)$ is of
length $d$ and is supported on the connected component of
$\Spec(\F_o\otimes_{\F_q} k)$ which is the image of $z$.
\item[(\textbf{M3})] If $x\neq o,\infty$, then
$$
\varphi_x(M_x)= M_x.
$$
\item[(\textbf{M4})] Some basis of $N$ generates $M_x$ in $N_x$ for all but finitely many
$x\in |X|$.
\end{enumerate}

\begin{defn} Let $\mathbf{DMod}$ be the category whose objects are
the pairs
$$((N,\varphi,\iota),
(M_x)_{x\in |X|})$$ where $(N,\varphi)$ is a $\varphi$-space of rank
$d^2$ over $F\otimes k$, $\iota:D^\opp \to \End(N,\varphi)$ is an
$F$-algebra homomorphism, and $(M_x)_{x\in |X|}$ is a collection of
$\cD_x$-lattices in $(N_x,\varphi_x)=(F_x\ctimes_F N, F_x\ctimes_F
\varphi)$ which satisfy (\textbf{M1})-(\textbf{M4}). A morphism
$\alpha$ between two such objects $$\alpha:((N',\varphi',\iota'),
(M_x')_{x\in |X|})\to ((N'',\varphi'',\iota''), (M_x'')_{x\in
|X|})
$$
is a morphism of the $\varphi$-spaces $\alpha:(N',\varphi')\to
(N'',\varphi'')$ such that $$\iota''\circ \alpha=\alpha\circ
\iota'\quad \text{ and }\quad \alpha\ctimes F_x (M_x')\subset
M_x''$$ for all $x\in |X|$.
\end{defn}

\begin{prop}\label{prop-latt} The functor $\mathbf{DES}\to
\mathbf{DMod}$ which associates to a $\cD$-elliptic sheaf of
characteristic $o$ over $k$ its generic fibre along with the
lattices $M_x$ in its Dieudonn\'e modules is an equivalence of
categories.
\end{prop}
\begin{proof}
From the description of a locally free sheaf on a curve through
lattices in its generic fibre, it follows that the functor in
question is fully faithful. Now let $((N,\varphi,\iota), (M_x)_{x\in
|X|})\in \mathbf{DMod}$ and $i\in \Z$. Define a sheaf $\cE_i$ on
$X\otimes_{\F_q}k$ as follows. Let $U\subset X$ be an open affine.
If $\infty\not\in |U|$, then let
$$
\cE_i(U\otimes_{\F_q} k):=\bigcap_{x\in |U|} (N\cap M_x),
$$
where the inner intersections are taken in $N_x$, and the outer in
$N$. If $\infty\in |U|$, let
$$
\cE_i(U\otimes_{\F_q} k):=\left(\bigcap_{x\in |U|-\infty} (N\cap
M_x)\right)\bigcap \left(\varphi_\infty^i(M_\infty)\cap N\right).
$$
Thanks to (\textbf{M4}), $\cE_i$ is a locally-free
$\cO_{X\otimes_{\F_q} k}$-module of rank $d^2$. The inclusions
$\varphi_\infty^i(M_\infty)\subset \varphi_\infty^{i+1}(M_\infty)$
induce inclusions $j_i:\cE_i\hookrightarrow \cE_{i+1}$. The action
of $\varphi$ on $N$ induces homomorphisms $t_i:\twist{\cE_i}\to
\cE_{i+1}$. The action of $D$ on $(N,\varphi)$ and $\cD_x$ on $M_x$,
defines an action of $\cD$ on $\cE_i$ compatible with $j_i$ and
$t_i$. Finally, the conditions (\textbf{M1})-(\textbf{M3}) ensure
that $(\cE_i,j_i,t_i)_{i\in \Z}\in \mathbf{DES}$. Hence our functor
is essentially surjective.
\end{proof}

Let $x\in |X|$ and $r:=[\F_x:\F_q]$. Since $N_x$ is a free
$F_x\ctimes_{\F_q}k$-module, by fixing an embedding $\F_x\to k$, we
obtain two actions of $\F_x$ on $N_x$. These actions induce a
grading
$$
N_x=\bigoplus_{i\in \Z/r\Z} N_{x,i},
$$
where $N_{x,i}=\{a\in N_x\ |\ (\la^{q^i}\ctimes 1)a=(1\ctimes
\la)a,\ \la\in \F_x\}$. Now $\varphi_x$ maps $N_{x,i}$ bijectively
into $N_{x,i+1}$, and $N_{x,0}$ is an $F_x\ctimes_{\F_x}k$-vector
space. Hence $(N_{x,0},\varphi_x^r)$ is a Dieudonn\'e $F_x$-module
over $k$ in the sense of $\S$\ref{sec1}. We can recover
$(N_x,\varphi_x)$ uniquely from $(N_{x,0},\varphi_x^r)$ since as
$F_x\ctimes_{\F_q}k$-vector space
$$
N_x\cong \bigoplus_{i\in \Z/r\Z} N_{x,0}
$$
with the action of $\varphi_x$ given by
$$
(a_0,a_1,\dots, a_{r-1})\mapsto (\varphi_x^r(a_{r-1}), a_0,\dots,
a_{r-2}).
$$
Finally, since the action of $D_x$ commutes with $\varphi_x$,
$(N_{x,0},\varphi_x^r)$ is a Dieudonn\'e $D_x$-module over $k$ and
$\End_{D_x}(N_x,\varphi_x)=\End_{D_x}(N_{x,0},\varphi_{x}^r)$.
Similar argument applies also to the lattices $M_x$ ($x\neq \infty$)
and produces Dieudonn\'e $\cD_x$-modules over $k$.

\section{Endomorphism rings}\label{Sec3} Let $D$ be as in
$\S$\ref{Sec2}. Assume $D$ is a division algebra such that $D_o$ is
the $d^2$-dimensional central division algebra over $F_o$ with
invariant $1/d$. In this case $\cD_o$ is the unique maximal order of
$D_o$ which we identify with $\F_o^{(d)}\comr{\Pi_o}$. Here
$\F_o^{(d)}$ is the degree $d$ extension of $\F_o$ and
\begin{align*}
\Pi_o a &=\Fr_q^{\deg(o)}(a)\Pi_o,\\
\Pi_o^d &=\pi_o.
\end{align*}

\begin{defn}\label{def3.1} Let $\E\in \mathbf{DES}$. We say that $\E$ is \textit{exceptional} if
$$\varphi_o^{\deg(o)}(M_o)=M_o\cdot\Pi_o.$$ Clearly $\E$ is
exceptional if and only if the Dieudonn\'e $\cD_o$-module
$(M_{o,0},\varphi_o^{\deg(o)})$ associated to $(M_o,\varphi_o)$ is
exceptional in the sense of Definition \ref{defnExp}. The
\textit{type} of an exceptional $\E$ is the type of
$(M_{o,0},\varphi_o^{\deg(o)})$. Similarly, we say that $\E$ is
\textit{special} (resp. \textit{superspecial}) if
$(M_{o,0},\varphi_o^{\deg(o)})$ is special (resp. superspecial).
\end{defn}

\begin{rem}
Exceptional $\cD$-elliptic sheaves do not correspond to points on
the moduli schemes constructed in \cite{Hausberger}, unless they are
superspecial.
\end{rem}

Let $\bar{D}$ be the central division algebra over $F$ with
invariants
\begin{equation}\label{eqbarD}
\inv_x(\bar{D})=\left\{
  \begin{array}{ll}
    1/d, & x=\infty; \\
    0, & x=o; \\
    \inv_x(D), & x\neq o,\infty.
  \end{array}
\right.
\end{equation}

\begin{thm}\label{thm2.2}
If $\E$ is exceptional of type $\bF$, then $\End(\E)$ has a natural
structure of a hereditary $\cO_X$-order in $\bar{D}$. This order is
maximal at every $x\in |X|-o$, and at $o$ it is isomorphic to
$\M_d(\bF,\cO_o)$.
\end{thm}
\begin{proof}
Let $((N,\varphi, \iota), (M_x)_{x\in |X|})\in \mathbf{DMod}$ be the
object attached to $\E$ by Proposition \ref{prop-latt}. Giving an
endomorphism of $\E$ is equivalent to giving
$$
\psi\in \End(N,\varphi,\iota)=\End_D(N,\varphi)
$$
such that $\psi\ctimes_F F_x\in \End(N_x,\varphi_x)$ preserves the
lattice $M_x$ for all $x\in |X|$.

Let $(\widetilde{F},\widetilde{\Pi})$ be the $\varphi$-pair of
$(N,\varphi)$; see \cite[App. A]{LRS} for the definition. Since $\E$
is exceptional
\begin{align*}
\varphi_\infty^{d\cdot\deg(\infty)}(M_\infty) &=
\pi_\infty^{-1}M_\infty,\\
\varphi_o^{d\cdot\deg(o)}(M_o) &=\pi_o M_o,\\
\varphi_x(M_x)&=M_x,\quad \text{if }x\neq o,\infty.
\end{align*}
Let $h$ be the class number of $X$. The divisor
$h(\deg(\infty)o-\deg(o)\infty)$ is principal, so from the previous
equalities $\varphi^{dh}\in F$. By construction of
$(\widetilde{F},\widetilde{\Pi})$, this implies $\widetilde{F}=F$
and $\widetilde{\Pi}\in F$ has valuations
\begin{equation}\label{eq-n}
\ord_x(\widetilde{\Pi})=\left\{
  \begin{array}{ll}
    1/d\deg(o), & x=o; \\
    -1/d\deg(\infty), & x=\infty; \\
    0, & x\neq o,\infty.
  \end{array}
\right.
\end{equation}
Since $(N,\varphi)$ is isotypical \cite[Lem. 9.6]{LRS}, (\ref{eq-n})
and \cite[Thm. A.6]{LRS} imply that $A=\End(N,\varphi)$ is the
cental simple algebra over $F$ of dimension $d^4$ with invariants
$\inv_o(A)=-1/d$, $\inv_\infty(A)=1/d$, $\inv_x(A)= 0$, $x\neq
o,\infty$. $\End_D(N,\varphi)$ is exactly the centralizer of
$\iota(D^\opp)$ in $\End(N,\varphi)$. By the double centralizer
theorem \cite[Cor. 7.14]{Reiner}
$$
\End_D(N,\varphi)\otimes_F D^\opp\cong A.
$$
This implies that $\End_D(N,\varphi)$ is the central simple algebra
over $F$ of dimension $d^2$ with invariants
$$
\inv_x(\End_D(N,\varphi))=\inv_x(A)-\inv_x(D^\opp)=\inv_x(A)+\inv_x(D)
\ \mod\ \Z.
$$
Comparing the invariants, we see that $\End_D(N,\varphi)\cong
\bar{D}$.

Giving an $\cO_X$-order $\cB$ in $\bar{D}$ is equivalent to giving a
set of $\cO_x$-orders $\cB_x\subset \bar{D}_x$ for all $x\in |X|$
such that there is an $F$-basis $B$ with $\cB_x=B\cdot \cO_x$ for
almost all $x$. Since a basis of $N$ spans almost all $M_x$, we
conclude that $\End(\E)$ is an $\cO_X$-order in $\bar{D}$. Let $x\in
|X|$. There is a natural morphism
\begin{equation}\label{eq1}
\End(\E)\otimes_{\cO_{X,x}}\cO_x \to \End_{\cD_x}(M_x,\varphi_x).
\end{equation}
Here the right hand-side denotes the subring of
$\End_{D_x}(N_x,\varphi_x)$ consisting of endomorphisms which
preserve $M_x$. By a standard property of the sheaf of local
morphisms, this morphism is injective with torsion-free cokernel. On
the other hand, by \cite[Lem. B.6-B.7]{LRS}
$$
\End_{D_x}(N_x,\varphi_x)\cong \bar{D}_x,
$$
if $x\neq o$. The same isomorphism for $x=o$ follows from
Proposition \ref{propEndM}. Hence (\ref{eq1}) becomes an isomorphism
after tensoring with $F_x$, and therefore is an isomorphism itself.
To finish the proof we need to show that
$\End_{\cD_x}(M_x,\varphi_x)$ is maximal for every $x\neq o$, and is
hereditary for $x=o$.

If $x=o$, then by Proposition \ref{propEndM}
$$\End_{\cD_x}(M_x,\varphi_x)\cong
\End_{\cD_x}(M_{x,0},\varphi_x^{\deg(x)})\cong \M_d(\bF, \cO_x).$$

If $x\neq o,\infty$, then $\varphi_x(M_x)=M_x$. By Proposition
\ref{lemEtale},
$$
(M_{x,0},\varphi_x^{\deg(x)})\cong (\La_x\ctimes_{\F_x}k,
\Id\ctimes_{\F_x}\Fr_q^{\deg(x)}),
$$ where $\La_x$ is a free $\cO_x$-module of
rank $d^2$. The action of $\cD_x$ commutes with $\varphi_x$, so
$\cD_x$ is in the right order of the full $\cO_x$-lattice $\La_x$ in
$D_x$. Since $\cD_x$ is maximal, the left order $O_l(\La_x)$ of
$\La_x$ is also maximal in $D_x\cong \bar{D}_x$; see
\cite[(17.6)]{Reiner}. On the other hand, $O_l(\La_x)\subset
\End_{\cD_x}(M_{x,0},\varphi_x^{\deg(x)})$, which forces
$\End_{\cD_x}(M_x,\varphi_x)$ to be maximal.

Finally, let $x=\infty$. By \cite[Lem. 9.8]{LRS},
$$(N_x,\varphi_x)\cong
(N_{d,-1},\varphi_{d,-1})^d,
$$
with the action of $\cD_x$ being the natural right action of
$\M_d(\cO_\infty)$. By definition, this action preserves $M_{x}$, so
Morita equivalence \cite[Ch. 4]{Reiner} reduces the problem to
showing the following:

In the notation of $\S$\ref{sec1}, if $M$ is a free $\cR$-module in
$N_{d,-1}$ such that $M\otimes_{R} K = N_{d,-1}$ and $M\subset
\varphi_{d,-1}(M)$, then $\End(M,\varphi_{d,-1})$ is a maximal order
in the central division algebra over $K$ with invariant $1/d$. Let
$M$ and $M'$ be any two such lattices. By \cite[Prop. B.10]{LRS},
$M'=\varphi_{d,-1}^n(M)$ for some $n\in \Z$, so
$\End(M,\varphi_{d,-1})=\End(M',\varphi_{d,-1})$. Therefore, we can
assume that $M$ is the free left $\cR$-module generated by
$1,\tau,\dots, \tau^{d-1}$ in
$\cK\{\tau\}/\cK\{\tau\}(\tau^d-\pi^{-1})$ with
$\varphi_{d,-1}=\tau$. Note that
$$
R_d\{\tau^{-1}\}/R_d\{\tau^{-1}\}(\tau^{-d}-\pi)\subset
\End(M,\tau)^\opp
$$
But the left hand-side is a maximal order, so $\End(M,\tau)$ is also
a maximal order.
\end{proof}

\begin{thm}\label{thm2.3} Assume $\E\in \mathbf{DES}$ is exceptional.
The map
$$
\E'\to \cI=\Hom(\E,\E')
$$
establishes a bijection between the set of isomorphism classes of
exceptional $\cD$-elliptic sheaves $\E'$ of the same type as $\E$
and the isomorphism classes of locally free rank-$1$ right
$\End(\E)$-modules. Under this bijection,
$$
\End(\E')\cong O_\ell(\cI).
$$
\end{thm}
\begin{proof} Denote $\cA:=\End(\E)$.
Let $\E'$ be an exceptional $\cD$-elliptic sheaf of the same type as
$\E$. Let $((N,\varphi, \iota), (M_x)_{x\in |X|})$ and
$((N',\varphi', \iota'), (M_x')_{x\in |X|})$ be the objects in
$\mathbf{DMod}$ attached to $\E$ and $\E'$, respectively, under the
equivalence of Proposition \ref{prop-latt}. From the proof of
Theorem \ref{thm2.2} we know that the $\varphi$-pairs
$(\widetilde{F},\widetilde{\Pi})$ associated to the generic fibres
of $\E$ and $\E'$ are the same. By \cite[(9.12)]{LRS}, this implies
that the generic fibres $(N,\varphi,\iota)$ and
$(N',\varphi',\iota')$ are isomorphic. (In the terminology of
\cite{LRS} this is equivalent to saying that $\E$ and $\E'$ are
isogenous.) Hence the Dieudonn\'e modules $(N_x,\varphi_x)$ and
$(N_x',\varphi_x')$ of $\E$ and $\E'$ are also isomorphic for all
$x\in |X|$. Consider the $\cD_x$-lattices $M_x\subset N_x$ and
$M_x'\subset N_x'$. We claim that there is an isomorphism
$\alpha_x:(N_x,\varphi_x)\cong (N_x',\varphi_x')$ which commutes
with $D_x$ and $\alpha(M_x)=M_x'$. When $x\neq o,\infty$, this
follows from Proposition \ref{lemEtale} and the fact that any two
maximal orders in $D_x$ are conjugate. When $x=o$, the claim follows
from the proof of Proposition \ref{propEndM}, using the assumption
that $\E$ and $\E'$ have the same type. Finally, when $x=\infty$
this follows from \cite[Prop. B.10]{LRS}. Moreover, thanks to
(\textbf{M4}), if we fix an isomorphism
$\alpha:(N,\varphi,\iota)\cong (N',\varphi',\iota')$, then for
almost all $x$ we can take $\alpha_x=\alpha\ctimes F_x$. The
argument which shows that (\ref{eq1}) is an isomorphism also implies
that for $x\in |X|$
$$
\Hom(\E,\E')\otimes_{\cO_{X,x}}{\cO_x}\cong
\Hom_{\cD_x}((M_x,\varphi_x), (M_x',\varphi_x')).
$$
Now from what was said above we conclude that $\Hom(\E,\E')$ is a
locally free rank-$1$ right $\cA$-module.

Conversely, let $\cI$ be a locally free rank-$1$ right $\cA$-module.
Define $(N', \varphi',\iota')=(N,\varphi,\iota)$ and
$M'_x=\cI_x\otimes_{\cA_x} M_x$, $x\in |X|$. It is easy to check
that the pair
$$
((N', \varphi',\iota'), (M'_x)_{x\in |X|})
$$
belongs to $\mathbf{DMod}$, hence defines a $\cD$-elliptic sheaf of
characteristic $o$ over $k$. We denote this $\cD$-elliptic sheaf by
$\cI\otimes_{\cA}\E$. Since $(M_o,\varphi_o)\cong
(M'_o,\varphi'_o)$, $\cI\otimes_{\cA}\E$ is exceptional of the same
type as $\E$.

There are natural morphisms
$$
\cI\to \Hom(\E,\cI\otimes_\cA \E)\quad\text{and}\quad
\Hom(\E,\E')\otimes_\cA \E \to \E',
$$
which are locally isomorphisms, so the two constructions are
inverses of each other, and the bijection of the theorem follows.

Finally, it is clear that $O_\ell(\cI)\subset \End(\cI\otimes_\cA
\E)$, and since both sides are locally isomorphic hereditary orders,
an equality must hold.
\end{proof}

\section{Mass-formula}\label{sMass}

Denote by $\fX_\bF$ the set of isomorphism classes of exceptional
$\cD$-elliptic sheaves of characteristic $o$ over $k$ of type $\bF$.
Using Proposition \ref{prop-latt}, one can easily show that
$\fX_\bF\neq \emptyset$; cf. the proof of \cite[Thm. 9.13]{LRS}. Let
$\E\in \fX_\bF$. Denote $\cA:=\End(\E)$. Let
$\bar{D}(\A_F):=\bar{D}\otimes_F \A_F$ and
$$
\cA(\A_F):=\prod_{x\in |X|}\cA_x\hookrightarrow \bar{D}(\A_F).
$$
The ring $\bar{D}(F)$ embeds diagonally into $\bar{D}(\A_F)$. One
consequence of Theorem \ref{thm2.3} is that there is a bijection
between $\fX_\bF$ and the double coset space
\begin{equation}\label{eq-dcs}
\bar{D}(F)^\times\bs \bar{D}(\A_F)^\times/\cA(\A_F)^\times.
\end{equation}
The order of this double coset space is infinite, which accounts for
the fact that there is a natural free action of $\Z$ on the category
of $\cD$-elliptic sheaves. In order to obtain ``finite'' spaces
while trying to classify $\cD$-elliptic sheaves, e.g. moduli schemes
of finite type over $F$, one has to mod out by this action as in
\cite{LRS}.

Let $(\cE_i, j_i,t_i)_{i\in \Z}$ be a $\cD$-elliptic sheaf. The
group $\Z$ acts by ``shifting the indices'': $$[n](\cE_i,
j_i,t_i)=(\cE_i', j_i',t_i')_{i\in \Z}$$ with $\cE_i'=\cE_{i+n}$,
$j_i'=j_{i+n}$, $t_i'=t_{i+n}$. It is clear that $\Z$ preserves the
set $\fX_\bF$.

Since $\bar{D}_\infty$ is a division algebra, the composition of the
reduced norm with the valuation at $\infty$ gives an isomorphism
$\bar{D}_\infty^\times/\cA_\infty^\times\cong \Z$.

\begin{lem}\label{lem5.1}
The action of $n\in \Z$ on the double coset space (\ref{eq-dcs})
corresponding to the action of $n$ on $\fX_\bF$ is the translation
by $n$ on $\bar{D}_\infty^\times/\cA_\infty^\times\cong \Z$.
\end{lem}
\begin{proof} Let $((N,\varphi, \iota), (M_x)_{x\in |X|})$ and
$((N',\varphi', \iota'), (M_x')_{x\in |X|})\in \mathbf{DMod}$ be the
objects attached to $\E$ and $\E'=[n]\cdot\E$, respectively, under
the equivalence of Proposition \ref{prop-latt}. Since the
restriction of $\cE_i$ to $(X-\infty)\otimes k$ does not depend on
$i$, $(N',\varphi')=(N,\varphi)$, $(M_x',
\varphi_x')=(M_x,\varphi_x)$ if $x\neq \infty$, and $(M_\infty',
\varphi_\infty')=(\varphi_\infty^n(M_\infty),\varphi_\infty)$. Let
$\Pi_\infty$ be a generator of the maximal ideal of $\cA_\infty$. We
conclude that the action of $1\in \Z$ on $\fX_\bF$ corresponds to
multiplication by $\Pi_\infty$ on the double coset space
(\ref{eq-dcs}). Since $\Pi_\infty$ maps to $1$ under the
homomorphism $\ord_\infty\circ \Nr: \bar{D}_\infty^\times\to \Z$,
the lemma follows.
\end{proof}

Form (\ref{eq-dcs}) and Lemma \ref{lem5.1}, we conclude that there
is a bijection between $\fX_\bF/\Z$ and the double coset space
\begin{equation}\label{eq-5.2}
\bar{D}(F)^\times\bs
\bar{D}(\A_F^\infty)^\times/\cA(\A_F^\infty)^\times,
\end{equation}
where $\bar{D}(\A_F^\infty)=\bar{D}\otimes_F \A_F^\infty$ and
$\cA(\A_F^\infty):=\prod_{x\in |X|-\infty}\cA_x$. By the strong
approximation theorem for $\bar{D}^\times$, this double coset space
has finite cardinality. Unfortunately, in general, an explicit
expression for class numbers of hereditary orders over Dedekind
domains is not known (e.g. the order of the double coset space
above). But one can at least give an estimate on this number using
an analogue of Eichler's mass-formula.

\vspace{0.1in}

Denote by $\cA^\times$ the sheaf of invertible elements in $\cA$.
Let
$$
\Aut(\E):=\G(X,\cA^\times).
$$
Let $R=\G(X-\infty, \cO_X)$ and $\La=\G(X-\infty, \cA)$. Then $\La$
is a hereditary $R$-order in $\bar{D}$. Clearly
$\La^\times=\G(X-\infty, \cA^\times)$.
\begin{lem}\label{lem6.2} The natural restriction map $\Aut(\E)\to \La^\times$ is
an isomorphism, and  $\La^\times\cong \F_{q^s}^\times$ for some $s$
dividing $d$.
\end{lem}
\begin{proof} First we show that $\La^\times\cong \F_{q^s}^\times$ for some $s$
dividing $d$. Let $Z$ be the center of $\bar{D}$ as an algebraic
group. Then $G=\bar{D}^\times/Z^\times$ is a projective algebraic
variety over $F$. Since $\bar{D}_\infty$ is a division algebra,
$G(F_\infty)$ is compact in $\infty$-adic topology, and contains
$\La^\times/R^\times$ as a discrete subgroup. Hence
$\La^\times/R^\times$ is finite, and as $R^\times\cong \F_q^\times$
is finite, $\La^\times$ is finite. Let $\la\in \La^\times$. Since
$\la^n=1$ for some $n$, $\la$ is algebraic over $\F_q$. Conversely,
it is clear that if $\la\in \La$ is algebraic over $\F_q$ and
$\la\neq 0$, then $\la\in \La^\times$. Let $\La^\alg$ be the subset
of $\La$ consisting of elements which are algebraic over $\F_q$. It
is easy to show that $\La^\alg$ is a field extension of $\F_q$; see
\cite[p. 383]{DvG}. Let $\La^\alg\cong \F_{q^s}$. Then $\F_{q^s}F$
is a field extension of $F$ of degree $s$ contained in $\bar{D}$.
This implies that $s$ divides $d$ (see \cite[Prop. A.1.4]{Laumon}),
so $\La^\times=\La^\alg-0=\F_{q^s}^\times$.

Let $\la\in \La^\times\subset \bar{D}\subset \bar{D}_\infty$. Since
$\bar{D}_\infty$ is a division algebra, $\cA_\infty$ is its unique
maximal order which is characterized as being the integral closure
of $\cO_\infty$ in $\bar{D}_\infty$. As $\la$ is obviously integral,
$\la\in \cA_\infty$, so $\la$ extends to a global section of
$\cA^\times$. This implies that $\Aut(\E)\to \La^\times$ is
surjective. It is clear that a section of $\cA^\times$ generically
generates a finite extension of $\F_q$. Hence if such a section is
$1$ on $X-\infty$, then it is identically $1$, so $\Aut(\E)\to
\La^\times$ is also injective.
\end{proof}

Let $\E_1,\dots, \E_h$ be representatives of $\fX_\bF/\Z$, and let
$w_i=\#\Aut(\E_i)$, $1\leq i\leq h$. Consider the sum
$$
\Mass(\bF):=(q-1)\sum_{i=1}^h \frac{1}{w_i}.
$$
The double coset space (\ref{eq-5.2}) is in bijection with
isomorphism classes of locally free rank-$1$ right $\La$-modules.
Let $I_1,\dots, I_h$ represent the isomorphism classes of such
modules. Let
$$\La_i=\G(X-\infty, \End(\E_i)),\quad 1\leq i\leq h.
$$
From Theorem \ref{thm2.3} one deduces that $O_\ell(I_i)=\La_i$.
Hence using Lemma \ref{lem6.2}
$$
\Mass(\bF)=\sum_{i=1}^h (\La_i^\times:R^\times)^{-1}.
$$
According to \cite{DvG}, it is possible to give a formula for this
last sum in terms of the invariants of $F$, $D$ and $\bF=(f_0,\dots,
f_{d-1})$. For $x\neq o$, $D_x\cong \M_{\kappa_x}(\Delta_x)$, where
$\Delta_x$ is a central division algebra over $F_x$ of index $e_x$.
We always have $\kappa_xe_x=d$, and $e_x=1$ if $x\not\in \Ram$. Let
$$
\cT^o=\prod_{\substack{x\in \Ram\\ x\neq o}}\prod_{\substack{1\leq j\leq d-1\\
j\not\equiv 0\ \mod\ e_x}}(q_x^j-1)
$$
and
$$
\cT_o=\frac{\prod_{1\leq j\leq d}(q_o^j-1)}{\prod_{0\leq i\leq
d-1}\prod_{1\leq j\leq f_i}(q_o^j-1)}.
$$
If we denote by $h(A)$ the class number of $A$, then \cite[(1)]{DvG}
specializes to
\begin{equation}\label{eq-mass}
\Mass(\bF)=h(A)\cdot \cT^o\cdot \cT_o\cdot\prod_{i=1}^{d-1}
\zeta_X(-i).
\end{equation}
From this we get our desired explicit estimate of the order of
(\ref{eq-5.2}):
$$
\Mass(\bF)\leq \#(\fX_\bF/\Z)\leq \frac{q^d-1}{q-1}\cdot \Mass(\bF).
$$

\vspace{0.1in}

We end this section with a geometric application of previous
results. Let $d=2$, and fix a closed finite subscheme $\fn\neq
\emptyset$ of $X-\infty-o$. Denote by $\Ell_{\cD,\fn}$ the modular
curve of $\cD$-elliptic sheaves which are special at $o$ in the
sense of \cite{Hausberger}, equipped with level-$\fn$ structures,
modulo the action of $\Z$.

\begin{rem}
The definition of $\cD$-elliptic sheaves in \cite{Hausberger}
includes a ``normalization'' condition which requires the
Euler-Poincar\'e characteristic of $\cE_0$ to be in the interval
$[0,d)$. The resulting category is equivalent to the quotient of the
category of $\cD$-elliptic sheaves by the action of $\Z$ as is done
in \cite{LRS}.
\end{rem}

According to Theorems 6.4 and 8.1 in \cite{Hausberger},
$\Ell_{\cD,\fn}$ is a fine moduli scheme which is projective of
relative dimension $1$ over $X'=(X-\fn-\infty-\Ram)\cup\{o\}$. It is
smooth over $X'$ except at $o$. The fibre of $\Ell_{\cD,\fn}$ over
$o$ is a reduced singular curve whose only singular points are
ordinary double points, and whose normalization is a disjoint union
of finitely many rational curves.

\begin{prop}\label{prop4.4} The singular points of $\Ell_{\cD,\fn}\times_{X'}\Spec(\overline{\F}_o)$
are represented by the isomorphisms classes of pairs
$(\E,\theta_\fn)$, where $\E$ is a superspecial $\cD$-elliptic sheaf
of characteristic $o$ over $k$, and $\theta_\fn$ is a level-$\fn$
structure on $\E$.
\end{prop}
\begin{proof} Denote
$Y:=\Ell_{\cD,\fn}\times_{X'}\Spec(\overline{\F}_o)$. As follows
from \cite[Prop. 4.1.2]{Genestier}, \cite[Prop. 2.16]{Hausberger}
and \cite[Prop. 2.1]{DrinfeldFS}, a $\cD$-elliptic sheaf of
characteristic $o$ over $k$ is special in the sense of Definition
\ref{def3.1} if and only if it is special in sense of \cite[Def.
3.5]{Hausberger}. Hence $Y$ classifies the pairs $(\E,\theta_\fn)$,
where $\E$ is a special $\cD$-elliptic sheaf over $k$ in the sense
of Definition \ref{def3.1}, and $\theta_\fn$ is a level-$\fn$
structure on $\E$.

Let $\cT:=\widehat{\Omega}^2\otimes_{\F_o} k$, where
$\widehat{\Omega}^2$ is the formal scheme over $\Spf(\cO_o)$
corresponding to Drinfeld's upper-half plane. By a theorem of
Drinfeld \cite{DrinfeldSym}, $\cT$ parametrizes special Dieudonn\'e
$\cD$-modules over $k$ (``special'' in the sense of Definition
\ref{defnExp}) equipped with some extra data. Proposition II.2.7.1
and the main result of Chapter III in \cite{Genestier} imply that
the singular points of $\cT$ correspond exactly to superspecial
Dieudonn\'e $\cD$-modules.

Finally, Hausberger's uniformization theorem \cite[Thm.
8.1]{Hausberger} relates $\cT$ and $Y$ as functors. This theorem,
combined with the previous two paragraphs, implies that a closed
point on $Y$ corresponding to $(\E,\theta_\fn)$ is singular if and
only if $\E$ is superspecial.
\end{proof}

Let $\cD_\fn=\cD\otimes_{\cO_X}(\cO_X/\cal{N})$, where $\cal{N}$ be
the ideal sheaf of $\fn$. $\F_q^\times$ embeds diagonally into
$\cD_n^\times$. Denote $d(\fn)=\#(\cD_\fn^\times/\F_q^\times)$.

\begin{cor}
The number of singular points on
$\Ell_{\cD,\fn}\times_{X'}\Spec(\overline{\F}_o)$ is equal to
$$
d(\fn)\cdot h(A)\cdot \zeta_X(-1)\cdot (q_o+1)\cdot \prod_{x\in
\Ram-o}(q_x-1).
$$
\end{cor}
\begin{proof}
A superspecial $\cD$-elliptic sheaf over $k$ is the same thing as an
exceptional $\cD$-elliptic sheaf of type $\bF=(1,1)$ (we assume
$d=2$). Fix such a $\cD$-elliptic sheaf $\E$ and let $w=\#\Aut(\E)$.
The number of all level-$\fn$ structures on $\E$, up to an
isomorphism, is equal to $d(\fn)\frac{q-1}{w}$. This combined with
Proposition \ref{prop4.4} implies that the number of singular points
on $\Ell_{\cD,\fn}\times_{X'}\Spec(\overline{\F}_o)$ is equal to
$d(\fn)\cdot\Mass(1,1)$. Now the corollary follows from
(\ref{eq-mass}).
\end{proof}

\section{Supersingular $\cD$-elliptic sheaves}\label{sLast} We keep the notation and assumptions of $\S$\ref{Sec2}.
In this section we assume $o\not \in \Ram$, i.e., $D_o\cong
\M_d(F_o)$.
\begin{defn}
$\E\in \mathbf{DES}$ is \textit{supersingular} if for all large
enough integers $n$
$$
\varphi_o^n(M_o)\subset \pi_oM_o.
$$
\end{defn}

Let $\bar{D}$ be the central division algebra over $F$ with
invariants
$$
\inv_x(\bar{D})\left\{
  \begin{array}{ll}
    1/d, & x=\infty; \\
    -1/d, & x=o; \\
    \inv_x(D), & x\neq o,\infty.
  \end{array}
\right.
$$

\begin{thm}\label{Thm2.6}
If $\E$ is a supersingular $\cD$-elliptic sheaf of characteristic
$o$ over $k$, then $\End(\E)$ is a maximal $\cO_X$-order in
$\bar{D}$. There is a bijection between the set of isomorphism
classes of supersingular $\cD$-elliptic sheaves over $k$ and the
isomorphism classes of locally free rank-$1$ right
$\End(\E)$-modules.
\end{thm}
\begin{proof}
The proof is mostly the same as the proof of Theorems \ref{thm2.2}
and \ref{thm2.3}. Two places where the argument needs to be slightly
modified are the following:

First, one shows that $\End(\E)$ is generically isomorphic to
$\bar{D}$ in this case by using the argument in the proof of
Proposition 9.9 and Corollary 9.10 in \cite{LRS}.

Second, $(N_{o,0},\varphi_o^{\deg(o)})\cong
(N_{d,1},\varphi_{d,1})^d$, by \cite[Lem. 9.8]{LRS}. Now to show
that $\End(\E)$ is maximal at $o$, one can proceed as in the
$x=\infty$ case of Theorem \ref{thm2.2}.
\end{proof}

Suppose $D=\M_d(F)$ and $\cD=\M_d(\cO_X)$. Let
$A=\G(X-\infty,\cO_X)$. Morita equivalence establishes an
equivalence between the category of $\cD$-elliptic sheaves of
characteristic $o$ over $k$, and the category of rank-$d$ elliptic
sheaves over $k$ with pole at $\infty$; cf. \cite[3.1.4]{Carayol}.
On the other hand, by a theorem of Drinfeld this latter category
modulo the action of $\Z$ is equivalent to the category of rank-$d$
Drinfeld $A$-modules over $k$; see \cite{DrinfeldES} and
\cite[$\S$2]{Carayol}. Hence the category of $\cD$-elliptic sheaves
over $k$ modulo the action of $\Z$ is equivalent to the category of
rank-$d$ Drinfeld $A$-modules over $k$. Under this equivalence,
supersingular $\cD$-elliptic sheaves correspond to supersingular
Drinfeld modules (see \cite{GekelerJA} for the definition of
supersingular Drinfeld modules). Indeed, the supersingular
$\cD$-elliptic sheaves and Drinfeld modules are uniquely
characterized by the fact that their endomorphism algebra is
$\bar{D}$. One concludes that in this case Theorem \ref{Thm2.6}
specializes to \cite[Thm. 4.3]{GekelerJA}.

% ------------------------------------------------------------------------

%\bibliographystyle{amsplain}
%\bibliography{Endom3}
\end{document}